\documentclass[11pt]{amsart}
\usepackage[T1]{fontenc}
\usepackage{amssymb}
\usepackage{enumitem}
\usepackage{mathrsfs}
\usepackage[colorlinks, linkcolor=blue, citecolor=blue, urlcolor=blue]{hyperref}
\usepackage{tikz}

\newtheorem{theorem}{Theorem}[section]
\newtheorem{corollary}[theorem]{Corollary}
\newtheorem{fact}[theorem]{Fact}
\newtheorem{lemma}[theorem]{Lemma}

\theoremstyle{definition}
\newtheorem{definition}[theorem]{Definition}
\newtheorem{question}[theorem]{Question}

\newcommand{\rstr}{{\upharpoonright}}

\DeclareMathOperator{\dom}{dom}
\DeclareMathOperator{\ran}{ran}
\DeclareMathOperator{\ns}{ns}
\DeclareMathOperator{\Part}{Part}
\DeclareMathOperator{\Partf}{Part_{fin}}
\DeclareMathOperator{\scrBf}{\mathscr{B}_{fin}}
\DeclareMathOperator{\fin}{fin}
\DeclareMathOperator{\seq}{seq}
\DeclareMathOperator{\seqi}{seq^{1-1}}
\DeclareMathOperator{\symg}{sym_{\mathcal{G}}}
\DeclareMathOperator{\fixg}{fix_{\mathcal{G}}}

\newcommand{\xdasharrow}[2][->]{
\tikz[baseline=-\the\dimexpr\fontdimen22\textfont2\relax]{
\node[anchor=south,font=\scriptsize, inner ysep=1.5pt,outer xsep=2.2pt](x){#2};
\draw[shorten <=3.4pt,shorten >=3.4pt,dashed,#1](x.south west)--(x.south east);
}
}

\newcommand{\xdash}[2][-]{
\tikz[baseline=-\the\dimexpr\fontdimen22\textfont2\relax]{
\node[anchor=south,font=\scriptsize, inner ysep=1.5pt,outer xsep=2.2pt](x){#2};
\draw[shorten <=3.4pt,shorten >=3.4pt,dashed,#1](x.south west)--(x.south east);
}
}

\begin{document}

\title[Cantor's theorem may fail for finitary partitions]{Cantor's theorem may fail for finitary partitions}

\author{Guozhen Shen}
\address{School of Philosophy\\
Wuhan University\\
No.~299 Bayi Road\\
Wuhan\\
Hubei Province 430072\\
People's Republic of China}
\email{shen\_guozhen@outlook.com}

\date{}

\begin{abstract}
A partition is finitary if all its members are finite.
For a set $A$, $\mathscr{B}(A)$ denotes the set of all finitary partitions of $A$.
It is shown consistent with $\mathsf{ZF}$ (without the axiom of choice) that
there exist an infinite set $A$ and a surjection from $A$ onto $\mathscr{B}(A)$.
On the other hand, we prove in $\mathsf{ZF}$ some theorems concerning $\mathscr{B}(A)$
for infinite sets $A$, among which are the following:
\begin{enumerate}
  \item If there is a finitary partition of $A$ without singleton blocks,
        then there are no surjections from $A$ onto $\mathscr{B}(A)$ and
        no finite-to-one functions from $\mathscr{B}(A)$ to $A$.
  \item For all $n\in\omega$, $|A^n|<|\mathscr{B}(A)|$.
  \item $|\mathscr{B}(A)|\neq|\mathrm{seq}(A)|$, where $\mathrm{seq}(A)$
        is the set of all finite sequences of elements of $A$.
\end{enumerate}
\end{abstract}

\subjclass[2020]{Primary 03E10; Secondary 03E25, 03E35}

\keywords{Cantor's theorem, axiom of choice, finitary partition, permutation model}

\maketitle

\section{Introduction}
In 1891, Cantor~\cite{Cantor1891} proved that, for all sets $A$,
there are no surjections from $A$ onto $\mathscr{P}(A)$ (the power set of~$A$).
Under the axiom of choice, for infinite sets $A$, several sets related to $A$
have the same cardinality as $\mathscr{P}(A)$ or $A$; for example, $\mathcal{S}(A)$ (the set of all permutations of~$A$)
and $\mathrm{Part}(A)$ (the set of all partitions of~$A$) have the same cardinality as $\mathscr{P}(A)$,
and $A^2$, $\mathrm{fin}(A)$ (the set of all finite subsets of~$A$), $\mathrm{seq}(A)$ (the set of all finite sequences
of elements of~$A$), and $\mathrm{seq}^{\text{1-1}}(A)$ (the set of all finite sequences without repetition of elements
of $A$) have the same cardinality as $A$. However, without the axiom of choice, this is no longer the case.
In 1924, Tarski~\cite{Tarski1924} proved that the statement that $A^2$ has the same cardinality as $A$
for all infinite sets $A$ is in fact equivalent to the axiom of choice.

Over the past century, various variations of Cantor's theorem have been investigated
in $\mathsf{ZF}$ (the Zermelo--Fraenkel set theory without the axiom of choice),
with $A$ or $\mathscr{P}(A)$ replaced by a set which has the same cardinality under the axiom of choice.
Specker~\cite{Specker1954} proves that, for all infinite sets $A$, there are no injections from $\mathscr{P}(A)$ into $A^2$.
Halbeisen and Shelah~\cite{HalbeisenShelah1994} prove that $|\mathrm{fin}(A)|<|\mathscr{P}(A)|$
and $|\mathrm{seq}^{\text{1-1}}(A)|\neq|\mathscr{P}(A)|\neq|\mathrm{seq}(A)|$.
Forster~\cite{Forster2003} proves that there are no finite-to-one functions from $\mathscr{P}(A)$ to $A$.
Recently, Peng and Shen~\cite{PengShen2022} prove that there are no surjections from $\omega\times A$ onto $\mathscr{P}(A)$,
and Peng, Shen and Wu~\cite{PengShenWu2022} prove that the existence of an infinite set $A$ and
a surjection from $A^2$ onto $\mathscr{P}(A)$ is consistent with $\mathsf{ZF}$.
The variations of Cantor's theorem with $\mathscr{P}(A)$ replaced by $\mathcal{S}(A)$ are
investigated in \cite{DawsonHoward1976,ShenYuan2020a,ShenYuan2020b,SonpanowVejjajiva2019}.

For a set $A$, let $\mathscr{B}(A)$ be the set of all finitary partitions of $A$,
where a partition is finitary if all its members are finite. We use the symbol $\mathscr{B}$
to denote this notion just because $|\mathscr{B}(n)|$ is the $n$-th Bell number.
The axiom of choice implies that $\mathscr{B}(A)$ and $\mathscr{P}(A)$ have the same cardinality for infinite sets $A$,
but each of ``$|\mathscr{B}(A)|<|\mathscr{P}(A)|$'', ``$|\mathscr{P}(A)|<|\mathscr{B}(A)|$'', and
``$|\mathscr{B}(A)|$ and $|\mathscr{P}(A)|$ are incomparable'' for some infinite set $A$ is consistent with $\mathsf{ZF}$.
Recently, Phansamdaeng and Vejjajiva~\cite{PhansamdaengVejjajiva2022} prove that $|\mathrm{fin}(A)|<|\mathscr{B}(A)|$
for all infinite sets $A$.

In this paper, we further study the variations of Cantor's theorem with $\mathscr{P}(A)$ replaced by $\mathscr{B}(A)$.
We prove that Cantor's theorem may fail for finitary partitions in the sense that the existence of an infinite set $A$
and a surjection from $A$ onto $\mathscr{B}(A)$ is consistent with $\mathsf{ZF}$. Nevertheless, we prove in $\mathsf{ZF}$
some theorems concerning $\mathscr{B}(A)$ for infinite sets $A$, among which are the following:
\begin{enumerate}
  \item If there is a finitary partition of $A$ without singleton blocks,
        then there are no surjections from $A$ onto $\mathscr{B}(A)$ and
        no finite-to-one functions from $\mathscr{B}(A)$ to $A$.
  \item For all $n\in\omega$, $|A^n|<|\mathscr{B}(A)|$.
  \item $|\mathscr{B}(A)|\neq|\mathrm{seq}(A)|$.
\end{enumerate}

\section{Some notation and preliminary results}
Throughout this paper, we shall work in $\mathsf{ZF}$.
In this section, we indicate briefly our use of some terminology and notation.
For a function $f$, we use $\dom(f)$ for the domain of $f$, $\ran(f)$ for the range of $f$,
$f[A]$ for the image of $A$ under $f$, $f^{-1}[A]$ for the inverse image of $A$ under $f$,
and $f\rstr A$ for the restriction of $f$ to $A$.
For functions $f$ and $g$, we use $g\circ f$ for the composition of $g$ and $f$.
We write $f:A\to B$ to express that $f$ is a function from $A$ to~$B$,
and $f:A\twoheadrightarrow B$ to express that $f$ is a function from $A$ \emph{onto} $B$.
For a set $A$, $|A|$ denotes the cardinality of $A$.

\begin{definition}
Let $A,B$ be arbitrary sets.
\begin{enumerate}
\item $|A|=|B|$, or $A\approx B$, if there is a bijection between $A$ and $B$.
\item $|A|\leqslant|B|$, or $A\preccurlyeq B$, if there is an injection from $A$ into $B$.
\item $|A|\leqslant^\ast|B|$, or $A\preccurlyeq^\ast B$, if there is a surjection from a subset of $B$ onto $A$.
\item $|A|<|B|$ if $|A|\leqslant|B|$ and $|A|\neq|B|$.
\end{enumerate}
\end{definition}

Clearly, if $A\preccurlyeq B$ then $A\preccurlyeq^\ast B$,
and if $A\preccurlyeq^\ast B$ then $\mathscr{P}(A)\preccurlyeq\mathscr{P}(B)$.

In the sequel, we shall frequently use expressions like ``one can explicitly define'' in our formulations,
which is illustrated by the following example.

\begin{theorem}[Cantor-Bernstein]\label{cbt}
From injections $f:A\to B$ and $g:B\to A$,
one can explicitly define a bijection $h:A\to B$.
\end{theorem}
\begin{proof}
See~\cite[Theorem 3.14]{Halbeisen2017}.
\end{proof}

\noindent
Formally, Theorem~\ref{cbt} states that one can define a class function $H$ without free
variables such that, whenever $f$ is an injection from $A$ into $B$ and $g$ is an injection from $B$ into $A$,
$H(f,g)$ is defined and is a bijection between $A$ and $B$.
Consequently, if $|A|\leqslant|B|$ and $|B|\leqslant|A|$, then $|A|=|B|$.

\begin{definition}
Let $A$ be a set and let $f$ be a function.
\begin{enumerate}
\item $A$ is \emph{Dedekind infinite} if $\omega\preccurlyeq A$; otherwise, $A$ is \emph{Dedekind finite}.
\item $A$ is \emph{power Dedekind infinite} if $\mathscr{P}(A)$ is Dedekind infinite; otherwise, $A$ is \emph{power Dedekind finite}.
\item $f$ is (\emph{Dedekind}) \emph{finite-to-one} if for every $z\in\ran(f)$, $f^{-1}[\{z\}]$ is (Dedekind) finite.
\end{enumerate}
\end{definition}

Clearly, if $f$ and $g$ are (Dedekind) finite-to-one functions, so is $g\circ f$ (cf.~\cite[Fact~2.8]{Shen2017}).
It is well-known that $A$ is Dedekind infinite if and only if there exists a bijection between $A$ and a proper subset of $A$.
For power Dedekind infinite sets, recall Kuratowski's celebrated theorem.

\begin{theorem}[Kuratowski]\label{kurt}
$A$ is power Dedekind infinite if and only if $\omega\preccurlyeq^\ast A$.
\end{theorem}
\begin{proof}
See~\cite[Proposition~5.4]{Halbeisen2017}.
\end{proof}

The following two facts are Corollaries~2.9 and~2.11 of \cite{Shen2017}, respectively.

\begin{fact}\label{sh01}
If $A$ is power Dedekind infinite and there exists a finite-to-one function from $A$ to $B$,
then $B$ is power Dedekind infinite.
\end{fact}

\begin{fact}\label{sh02}
If $A^n$ is power Dedekind infinite, so is $A$.
\end{fact}

Let $P$ be a partition of $A$.
We say that $P$ is \emph{finitary} if all blocks of $P$ are finite,
and write $\ns(P)$ for the set of non-singleton blocks of $P$.
For $x\in A$, we write $[x]_P$ for the unique block of $P$ which contains $x$.
The equivalence relation $\sim_P$ on $A$ induced by $P$ is defined by
\[
x\sim_Py\qquad\text{if and only if}\qquad[x]_P=[y]_P.
\]

\begin{definition}
Let $A$ be an arbitrary set.
\begin{enumerate}
\item $\Part(A)$ is the set of all partitions of $A$.
\item $\Partf(A)=\{P\in\Part(A)\mid P\text{ is finite}\}$.
\item $\mathscr{B}(A)=\{P\in\Part(A)\mid P\text{ is finitary}\}$.
\item $\scrBf(A)=\{P\in\mathscr{B}(A)\mid\ns(P)\text{ is finite}\}$.
\item $\fin(A)$ is the set of all finite subsets of $A$.
\item $\seq(A)=\{f\mid f\text{ is a function from an }n\in\omega\text{ to }A\}$.
\item $\seqi(A)=\{f\mid f\text{ is an injection from an }n\in\omega\text{ into }A\}$.
\end{enumerate}
\end{definition}

Below we list some basic relations between the cardinalities of these sets.
We first note that $\fin(A)\preccurlyeq^\ast\seqi(A)\preccurlyeq\seq(A)$.
The next three facts are Facts~2.13, 2.16, and~2.17 of \cite{Shen2017}, respectively.

\begin{fact}\label{sh03}
If $A$ is infinite, then $\fin(A)$ and $\mathscr{P}(A)$ are power Dedekind infinite.
\end{fact}

\begin{fact}\label{sh04}
$\seqi(A)\preccurlyeq\fin(\fin(A))$.
\end{fact}

\begin{fact}\label{sh05}
There is a finite-to-one function from $\fin(\fin(A))$ to $\fin(A)$.
\end{fact}

The next three facts are Facts~2.19, 2.20, and Corollary~2.23 of \cite{ShenYuan2020a}, respectively.

\begin{fact}\label{sh06}
If $A$ is non-empty, then $\seq(A)$ is Dedekind infinite.
\end{fact}

\begin{fact}\label{sh07}
If $A$ is Dedekind finite, then there is a Dedekind finite-to-one function from $\seq(A)$ to $\omega$.
\end{fact}

\begin{fact}\label{sh08}
If $A$ is Dedekind infinite, then $\seq(A)\approx\seqi(A)$.
\end{fact}

\begin{fact}\label{sh09}
$\Part(A)\preccurlyeq\mathscr{P}(A^2)$.
\end{fact}
\begin{proof}
The function that maps each partition $P$ of $A$ to $\sim_P$
is an injection from $\Part(A)$ into $\mathscr{P}(A^2)$.
\end{proof}

\begin{corollary}\label{sh10}
If $A$ is power Dedekind finite, then $\Part(A)$, and hence also $\Partf(A)$ and $\mathscr{B}(A)$, are Dedekind finite.
\end{corollary}
\begin{proof}
If $A$ is power Dedekind finite, so is $A^2$ by Fact~\ref{sh02}, and thus $\mathscr{P}(A^2)$ is Dedekind finite,
which implies that also $\Part(A)$, $\Partf(A)$ and $\mathscr{B}(A)$ are Dedekind finite by Fact~\ref{sh09}.
\end{proof}

\begin{fact}\label{sh11}
$\scrBf(A)\preccurlyeq\fin(\fin(A))$.
\end{fact}
\begin{proof}
The function that maps each $P\in\scrBf(A)$ to $\ns(P)$
is an injection from $\scrBf(A)$ into $\fin(\fin(A))$.
\end{proof}

\begin{fact}\label{sh14}
$\scrBf(A)\preccurlyeq\Partf(A)$.
\end{fact}
\begin{proof}
If $A$ is finite, then $\scrBf(A)=\Partf(A)$; otherwise, the function
that maps each $P\in\scrBf(A)$ to $\ns(P)\cup\{\bigcup(P\setminus\ns(P))\}$
is an injection from $\scrBf(A)$ into $\Partf(A)$.
\end{proof}

\begin{fact}\label{sh12}
If $|A|\geqslant5$, then $\fin(A)\preccurlyeq\scrBf(A)$ and $\mathscr{P}(A)\preccurlyeq\Partf(A)$.
\end{fact}
\begin{proof}
Let $E=\{a,b,c,d,e\}$ be a $5$-element subset of $A$.
We define functions $f:\fin(A)\to\scrBf(A)$ and $g:\mathscr{P}(A)\to\Partf(A)$ by setting, for $B\in\fin(A)$ and $C\in\mathscr{P}(A)$,
\[
f(B)=
\begin{cases}
(\{B\}\cup[A\setminus B]^1)\setminus\{\varnothing\}      & \text{if $B$ is not a singleton,}\\
\{\{a,x\},E\setminus\{a,x\}\}\cup[A\setminus(B\cup E)]^1 & \text{if $B=\{x\}$ for some $x\neq a$,}\\
\{\{a\},\{b,c\},\{d,e\}\}\cup[A\setminus E]^1            & \text{if $B=\{a\}$,}
\end{cases}
\]
and
\[
g(C)=
\begin{cases}
\{C,A\setminus C\}\setminus\{\varnothing\} & \text{if $a\notin C$,}\\
(\{C,A\setminus(C\cup E)\}\cup[E\setminus C]^1)\setminus\{\varnothing\} & \text{if $a\in C$, $|C\cap E|\leqslant3$,}\\
\{C\setminus\{a\},\{a\},E\setminus C,A\setminus(C\cup E)\}\setminus\{\varnothing\} & \text{if $a\in C$, $|C\cap E|=4$,}\\
\{C\setminus\{b,c,d,e\},\{b,c\},\{d,e\},A\setminus C\}\setminus\{\varnothing\} & \text{if $E\subseteq C$,}
\end{cases}
\]
where $[D]^1$ denotes the set of $1$-element subsets of $D$.
It is easy to see that $f$ and $g$ are injective.
\end{proof}

The following corollary immediately follows from Facts~\ref{sh03} and~\ref{sh12}.

\begin{corollary}\label{sh13}
If $A$ is infinite, then $\scrBf(A)$ and $\mathscr{B}(A)$ are power Dedekind infinite.
\end{corollary}

For infinite sets $A$, the relations between the cardinalities of $\fin(A)$, $\mathscr{P}(A)$,
$\scrBf(A)$, $\mathscr{B}(A)$, $\Partf(A)$, $\Part(A)$, and $\mathscr{P}(A^2)$ can be visualized
by the following diagram (where for two sets $X$ and $Y$, $X\longrightarrow Y$ means $|X|\leqslant|Y|$,
$X\xdasharrow{\hspace*{4mm}}Y$ means $|X|<|Y|$, and $X\xdash{\hspace*{4mm}}Y$ means $|X|\neq|Y|$).

\begin{figure}[htb]
\renewcommand{\figurename}{Diagram}
\begin{tikzpicture}
  \node (a) at (0,0)  {$\fin(A)$};
  \node (b) at (-3,2) {$\scrBf(A)$};
  \node (c) at (-3,4) {$\mathscr{B}(A)$};
  \node (d) at (3,2)  {$\mathscr{P}(A)$};
  \node (e) at (3,4)  {$\Partf(A)$};
  \node (f) at (0,6)  {$\Part(A)$};
  \node (g) at (0,8)  {$\mathscr{P}(A^2)$};
  \draw[->] (a)--(b);
  \draw[->] (b)--(c);
  \draw[->] (c)--(f);
  \draw[->] (f)--(g);
  \draw[->] (d)--(e);
  \draw[->] (e)--(f);
  \draw[->,dashed] (a)--(c);
  \draw[->,dashed] (a)--(d);
  \draw[->,dashed] (b)--(e);
  \draw[dashed] (b)--(d);
\end{tikzpicture}
\caption{Relations between the cardinalities of some sets}\label{sh15}
\end{figure}
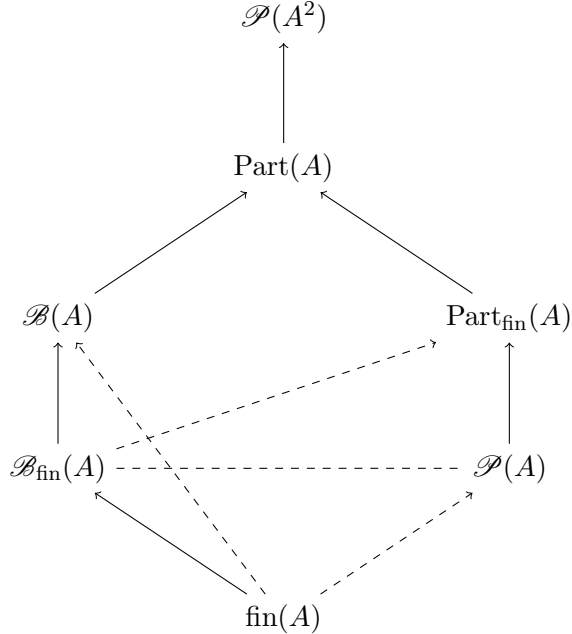

In the above diagram, the $\leqslant$-relations have already been established,
and the inequalities $|\mathrm{fin}(A)|<|\mathscr{P}(A)|$ and $|\mathrm{fin}(A)|<|\mathscr{B}(A)|$
are proved in~\cite[Theorem~3]{HalbeisenShelah1994} and~\cite[Theorem~3.7]{PhansamdaengVejjajiva2022}, respectively.
The inequalities $|\mathscr{B}_{\mathrm{fin}}(A)|<|\mathrm{Part}_{\mathrm{fin}}(A)|$ and
$|\mathscr{B}_{\mathrm{fin}}(A)|\neq|\mathscr{P}(A)|$ will be proved in Section~\ref{sh00}.
The other relations not indicted in the diagram cannot be proved in~$\mathsf{ZF}$, as shown in the next section.

\section{Permutation models and consistency results}
We refer the readers to~\cite[Chap.~8]{Halbeisen2017} or~\cite[Chap.~4]{Jech1973}
for an introduction to the theory of permutation models.
Permutation models are not models of $\mathsf{ZF}$;
they are models of $\mathsf{ZFA}$ (the Zermelo--Fraenkel set theory with atoms).
Nevertheless, they indirectly give, via the Jech--Sochor theorem
(cf.~\cite[Theorem~17.2]{Halbeisen2017} or~\cite[Theorem~6.1]{Jech1973}), models of $\mathsf{ZF}$.

Let $A$ be the set of atoms, let $\mathcal{G}$ be a group of permutations of $A$,
and let $\mathfrak{F}$ be a normal filter on $\mathcal{G}$.
We write $\symg(x)$ for the set $\{\pi\in\mathcal{G}\mid\pi x=x\}$,
where $\pi\in\mathcal{G}$ extends to a permutation of the universe by
\[
\pi x=\{\pi y\mid y\in x\}.
\]
Then $x$ belongs to the permutation model $\mathcal{V}$ determined by $\mathcal{G}$ and $\mathfrak{F}$
if and only if $x\subseteq\mathcal{V}$ and $\symg(x)\in\mathfrak{F}$.

For each $E\subseteq A$, we write $\fixg(E)$ for the set $\{\pi\in\mathcal{G}\mid\forall a\in E(\pi a=a)\}$.
Let $\mathcal{I}\subseteq\mathscr{P}(A)$ be a normal ideal and let $\mathfrak{F}$ be the normal filter on $\mathcal{G}$
generated by the subgroups $\{\fixg(E)\mid E\in\mathcal{I}\}$. Then $x$ belongs to the permutation model $\mathcal{V}$
determined by $\mathcal{G}$ and $\mathcal{I}$ if and only if $x\subseteq\mathcal{V}$ and there exists an $E\in\mathcal{I}$
such that $\fixg(E)\subseteq\symg(x)$; that is, every $\pi\in\mathcal{G}$ fixing $E$ pointwise also fixes $x$.
Such an $E$ is called a \emph{support} of $x$.

\subsection{A model for $|\mathscr{B}(A)|\leqslant^\ast|A|$ and $|\mathscr{B}(A)|<|\mathscr{P}(A)|$}
We construct a permutation model $\mathcal{V}_\mathscr{B}$ in which the set $A$ of atoms satisfies
$|\mathscr{B}(A)|\leqslant^\ast|A|$ and $|\mathscr{B}(A)|<|\mathscr{P}(A)|$.
The atoms are constructed by recursion as follows:
\begin{enumerate}[label=\upshape(\roman*), leftmargin=*, widest=iii]
  \item $A_0=\varnothing$ and $\mathcal{G}_0=\{\varnothing\}$ is the group of all permutations of $A_0$.
  \item $A_{n+1}=A_n\cup\{(n,P,k)\mid P\in\mathscr{B}(A_n)\text{ and }k\in\omega\}$.
  \item $\mathcal{G}_{n+1}$ is the group of permutations of $A_{n+1}$ consisting of
        all permutations $h$ for which there exists a $g\in\mathcal{G}_n$ such that
        \begin{itemize}[leftmargin=*]
          \item $g=h\rstr A_n$;
          \item for each $P\in\mathscr{B}(A_n)$, there exists a permutation $q$ of $\omega$ such that
                $h(n,P,k)=(n,\{g[D]\mid D\in P\},q(k))$ for all $k\in\omega$.
        \end{itemize}
\end{enumerate}
Let $A=\bigcup_{n\in\omega}A_n$ be the set of atoms, let $\mathcal{G}$ be the group of permutations of $A$
consisting of all permutations $\pi$ such that $\pi\rstr A_n\in\mathcal{G}_n$ for all $n\in\omega$,
and let $\mathfrak{F}$ be the normal filter on $\mathcal{G}$ generated by the subgroups $\{\fixg(A_n)\mid n\in\omega\}$.
The permutation model determined by $\mathcal{G}$ and $\mathfrak{F}$ is denoted by $\mathcal{V}_\mathscr{B}$.

\begin{lemma}\label{sh16}
For every $P\in\mathscr{B}(A)$, $P\in\mathcal{V}_\mathscr{B}$ if and only if $\ns(P)\subseteq\mathscr{P}(A_m)$ for some $m\in\omega$.
\end{lemma}
\begin{proof}
Let $P\in\mathscr{B}(A)$. If $\ns(P)\subseteq\mathscr{P}(A_m)$ for some $m\in\omega$, then clearly $\fixg(A_m)\subseteq\symg(P)$,
which implies that $P\in\mathcal{V}_\mathscr{B}$. For the other direction, suppose $P\in\mathcal{V}_\mathscr{B}$ and let
$m\in\omega$ be such that $\fixg(A_m)\subseteq\symg(P)$; that is, every $\pi\in\mathcal{G}$ fixing $A_m$ pointwise also fixes $P$.
We claim $\ns(P)\subseteq\mathscr{P}(A_m)$. Assume towards a contradiction that $x\sim_Py$ for some distinct $x,y$ such that
one of $x$ and $y$ is not in $A_m$. Suppose that $x=(n,Q,k)$ and $y=(n',Q',k')$, and assume without loss of generality $n'\leqslant n$.
Then $x\notin A_m$ and thus $m\leqslant n$. Let $l\in\omega$ be such that $(n,Q,l)\notin[y]_P$ and let $q$ be the transposition that
swaps $k$ and $l$. Since $P$ is finitary, such an $l$ exists. Let $h$ be the permutation of $A_{n+1}$ such that $h$ fixes $A_n$ pointwise
and for all $R\in\mathscr{B}(A_n)$ and all $j\in\omega$, $h(n,R,j)=(n,R,q(j))$ if $R=Q$, and $h(n,R,j)=(n,R,j)$ otherwise.
Then $h\in\mathcal{G}_{n+1}$ fixes $A_{n+1}\setminus\{x,(n,Q,l)\}$ pointwise. Hence $h(y)=y$.
Extend $h$ in a straightforward way to some $\pi\in\mathcal{G}$. Then $\pi\in\fixg(A_m\cup\{y\})$
and $\pi(x)=(n,Q,l)\notin[y]_P$. Thus $\pi$ moves $P$, which is a contradiction.
\end{proof}

\begin{lemma}\label{sh17}
In $\mathcal{V}_\mathscr{B}$, $|\mathscr{B}(A)|\leqslant^\ast|A|$ and $|\mathscr{B}(A)|<|\mathscr{P}(A)|$.
\end{lemma}
\begin{proof}
Let $\Phi$ be the function on $\{P\in\mathscr{B}(A)\mid\exists m\in\omega(\ns(P)\subseteq\mathscr{P}(A_m))\}$ defined by
\[
\Phi(u)=\{(n_P,P\cap\mathscr{P}(A_{n_P}),k)\mid k\in\omega\},
\]
where $n_P$ is the least $m\in\omega$ such that $\ns(P)\subseteq\mathscr{P}(A_m)$. Clearly, $\Phi\in\mathcal{V}_\mathscr{B}$.
In~$\mathcal{V}_\mathscr{B}$, by Lemma~\ref{sh16}, $\Phi$ is an injection from $\mathscr{B}(A)$ into $\mathscr{P}(A)$,
and the sets in the range of $\Phi$ are pairwise disjoint, which implies that $|\mathscr{B}(A)|\leqslant^\ast|A|$.
Since $|\mathscr{P}(A)|\nleqslant^\ast|A|$ by Cantor's theorem, it follows that $|\mathscr{B}(A)|<|\mathscr{P}(A)|$.
\end{proof}

Now the next theorem immediately follows from Lemma~\ref{sh17} and the Jech--Sochor theorem.

\begin{theorem}\label{sh18}
It is consistent with $\mathsf{ZF}$ that there exists an infinite set $A$
for which $|\mathscr{B}(A)|\leqslant^\ast|A|$ and $|\mathscr{B}(A)|<|\mathscr{P}(A)|$.
\end{theorem}

\subsection{A model for $|\mathscr{P}(A)|<|\mathscr{B}(A)|<|\mathrm{Part}_{\mathrm{fin}}(A)|$}
We show that the ordered Mostowski model $\mathcal{V}_\mathrm{M}$ (cf.~\cite[pp.~198--202]{Halbeisen2017} or~\cite[\S4.5]{Jech1973})
is a model of this kind. Recall that the set $A$ of atoms carries an ordering $<_\mathrm{M}$ which is
isomorphic to the ordering of the rational numbers, the permutation group $\mathcal{G}$ consists of all automorphisms
of $\langle A,<_\mathrm{M}\rangle$, and $\mathcal{V}_\mathrm{M}$ is determined by $\mathcal{G}$ and finite supports.
Clearly, the ordering $<_\mathrm{M}$ belongs to $\mathcal{V}_\mathrm{M}$ (cf.~\cite[Lemma~8.10]{Halbeisen2017}).
In $\mathcal{V}_\mathrm{M}$, $A$ is infinite but power Dedekind finite (cf.~\cite[Lemma~8.13]{Halbeisen2017}),
and thus, by Fact~\ref{sh02} and Corollary~\ref{sh10}, $\mathscr{P}(A^2)$ and $\mathscr{B}(A)$ are Dedekind finite.

\begin{lemma}\label{sh19}
In $\mathcal{V}_\mathrm{M}$, $\mathscr{B}(A)=\scrBf(A)$.
\end{lemma}
\begin{proof}
Let $P\in\mathcal{V}_\mathrm{M}$ be a finitary partition of $A$ and let $E$ be a finite support of $P$.
We claim $\ns(P)\subseteq\mathscr{P}(E)$. Assume towards a contradiction that $x\sim_Py$ for some distinct $x,y$ such that $x\notin E$.
Since $P$ is finitary, we can find a $\pi\in\fixg(E\cup\{y\})$ such that $\pi(x)\notin[y]_P$.
Hence $\pi$ moves $P$, contradicting that $E$ is a support of $P$.
Thus $\ns(P)\subseteq\mathscr{P}(E)$, so $P\in\scrBf(A)$.
\end{proof}

The next two lemmas are Lemmas~8.11(b) and~8.12 of~\cite{Halbeisen2017}, respectively.

\begin{lemma}\label{sh20}
Every $x\in\mathcal{V}_\mathrm{M}$ has a least support.
\end{lemma}

\begin{lemma}\label{sh21}
If $E$ is an $n$-element subset of $A$, then $E$ supports exactly $2^{2n+1}$ subsets of $A$.
\end{lemma}

\begin{lemma}\label{sh22}
For each $n\in\omega$, let $B_n^\star$ be the number of partitions of $n$ without singleton blocks; that is,
\[
B_n^\star=|\{P\in\mathscr{B}(n)\mid\ns(P)=P\}|.
\]
If $n\geqslant23$, then $2^{2n+2}<B_n^\star$.
\end{lemma}
\begin{proof}
For each $n\in\omega$, let $B_n$ be the $n$-th Bell number; that is, $B_n=|\mathscr{B}(n)|$.
Recall Dobinski's formula (see, for example,~\cite{Rota1964}):
\[
B_n=\frac{1}{e}\sum_{k=0}^\infty\frac{k^n}{k!}.
\]
It is easy to see that $B_n=B_n^\star+B_{n+1}^\star$. Hence, for $n\geqslant23$, we have
\[
B_n^\star>\frac{B_{n-1}}{2}>\frac{8^{n-1}}{2e\cdot8!}=\frac{2^{n-5}}{2e\cdot8!}\cdot2^{2n+2}
\geqslant\frac{2^{18}}{2e\cdot8!}\cdot2^{2n+2}>2^{2n+2}.\qedhere
\]
\end{proof}

\begin{lemma}\label{sh23}
In $\mathcal{V}_\mathrm{M}$, $|\mathscr{P}(A)|<|\mathscr{B}(A)|<|\mathrm{Part}_{\mathrm{fin}}(A)|<|\mathrm{Part}(A)|<|\mathscr{P}(A^2)|$.
\end{lemma}
\begin{proof}
In $\mathcal{V}_\mathrm{M}$, since $\mathscr{P}(A^2)$ is Dedekind finite, and the injections constructed
in the proofs of Facts~\ref{sh09} and~\ref{sh14} are clearly not surjective, it follows that
$|\mathscr{B}_{\mathrm{fin}}(A)|<|\mathrm{Part}_{\mathrm{fin}}(A)|<|\mathrm{Part}(A)|<|\mathscr{P}(A^2)|$.
By Lemma~\ref{sh19}, it remains to show that $|\mathscr{P}(A)|<|\mathscr{B}(A)|$.

For a finite subset $E$ of $A$, we can use $<_\mathrm{M}$ to define an ordering of the subsets of $A$ supported by $E$
and an ordering of the finitary partitions $P$ of $A$ with $\ns(P)\subseteq\mathscr{P}(E)$.
Let $D=\{a_i\mid i<46\}$ be a $46$-element subset of $A$. In $\mathcal{V}_\mathrm{M}$,
we define an injection $f$ from $\mathscr{P}(A)$ into $\mathscr{B}(A)$ as follows.

Let $C\in\mathcal{V}_\mathrm{M}$ be a subset of $A$. By Lemma~\ref{sh20}, $C$ has a least support $E$.
Suppose that $C$ is the $k$-th subset of $A$ with $E$ as its least support.
Let $n=|D\bigtriangleup E|$, where $\bigtriangleup$ denotes the symmetric difference.
Now, if $|E|\geqslant23$, define $f(C)$ to be the $k$-th finitary partition $P$ of $A$ with $\bigcup\ns(P)=E$;
otherwise, define $f(C)$ to be the $(B_n^\star-k-1)$-th finitary partition $P$ of $A$ with $\bigcup\ns(P)=D\bigtriangleup E$.
In the second case, $n\geqslant23$, and thus $B_n^\star-k-1>2^{2n+1}$ by Lemmas~\ref{sh21} and~\ref{sh22}.
Thus $f$ is injective. Since $D$ is a finite support of~$f$, it follows that $f\in\mathcal{V}_\mathrm{M}$.

Finally, since $\mathscr{B}(A)$ is Dedekind finite and $f$ is not surjective, it follows that $|\mathscr{P}(A)|<|\mathscr{B}(A)|$.
\end{proof}

Now the next theorem immediately follows from Lemmas~\ref{sh19} and~\ref{sh23} and the Jech--Sochor theorem.

\begin{theorem}\label{sh24}
It is consistent with $\mathsf{ZF}$ that there exists an infinite set $A$ for which $\mathscr{B}(A)=\scrBf(A)$
and $|\mathscr{P}(A)|<|\mathscr{B}(A)|<|\mathrm{Part}_{\mathrm{fin}}(A)|<|\mathrm{Part}(A)|<|\mathscr{P}(A^2)|$.
\end{theorem}

\subsection{A model in which $\mathscr{B}(A)$ is incomparable with $\mathscr{P}(A)$ or $\Partf(A)$}
We use a variation of the basic Fraenkel model (cf.~\cite[pp.~195--196]{Halbeisen2017} or~\cite[\S4.3]{Jech1973}).
Let $A$ be an uncountable set of atoms, let $\mathcal{G}$ be the group of all permutations of $A$,
and let $\mathcal{V}_\mathrm{F}$ be the permutation model determined by $\mathcal{G}$ and countable supports.

\begin{lemma}\label{sh25}
In $\mathcal{V}_\mathrm{F}$, $|\mathscr{B}(A)|\nleqslant|\mathrm{Part}_{\mathrm{fin}}(A)|$,
$|\mathscr{B}_{\mathrm{fin}}(A)|\nleqslant|\mathscr{P}(A)|$, and $|\mathscr{P}(A)|\nleqslant|\mathscr{B}(A)|$.
\end{lemma}
\begin{proof}
(1) $|\mathscr{B}(A)|\nleqslant|\mathrm{Part}_{\mathrm{fin}}(A)|$.
Assume towards a contradiction that in $\mathcal{V}_\mathrm{F}$ there is an injection $f$
from $\mathscr{B}(A)$ into $\Partf(A)$. Let $B$ be a countable support of $f$.
Let $\{a_n\mid n\in\omega\}\subseteq A\setminus B$ with $a_i\neq a_j$ whenever $i\neq j$.
Consider the finitary partition $P=\{\{a_{2n},a_{2n+1}\}\mid n\in\omega\}\cup[A\setminus\{a_n\mid n\in\omega\}]^1\in\mathcal{V}_\mathrm{F}$.
Since $f(P)$ is a finite partition, there must be $i,j\in\omega$ with $i\neq j$ such that
$a_{2i}$ and $a_{2j}$ are in the same block of $f(P)$.
The transposition that swaps $a_{2i}$ and $a_{2i+1}$ fixes $P$, and thus also fixes $f(P)$,
which implies that $a_{2i+1}$ and $a_{2j}$ are also in the same block of $f(P)$.
Hence, the transposition that swaps $a_{2i+1}$ and $a_{2j}$ fixes $f(P)$,
but it moves $P$, contradicting that $f$ is injective.

(2) $|\mathscr{B}_{\mathrm{fin}}(A)|\nleqslant|\mathscr{P}(A)|$.
Assume towards a contradiction that in $\mathcal{V}_\mathrm{F}$ there is an injection $g$
from $\scrBf(A)$ into $\mathscr{P}(A)$. Let $C$ be a countable support of~$g$.
Let $a_0,a_1,a_2,a_3$ be four distinct elements of $A\setminus C$.
Consider the finitary partition $P=\{\{a_0,a_1\},\{a_2,a_3\}\}\cup[A\setminus\{a_0,a_1,a_2,a_3\}]^1\in\mathcal{V}_\mathrm{F}$.
Clearly, for any $i,j<4$ with $i\neq j$, there is a $\pi\in\fixg(C)$ such that $\pi(P)=P$ and $\pi(a_i)=a_j$.
Hence, $\{a_0,a_1,a_2,a_3\}\subseteq g(P)$ or $\{a_0,a_1,a_2,a_3\}\subseteq A\setminus g(P)$.
Thus, the transposition that swaps $a_0$ and $a_2$ fixes $g(P)$,
but it moves $P$, contradicting that $g$ is injective.

(3) $|\mathscr{P}(A)|\nleqslant|\mathscr{B}(A)|$.
Assume towards a contradiction that in $\mathcal{V}_\mathrm{F}$ there is an injection $h$
from $\mathscr{P}(A)$ into $\mathscr{B}(A)$. Let $D$ be a countable support of $h$.
Let $C$ be a denumerable subset of $A\setminus D$. Take $x\in C$ and $y\in A\setminus(C\cup D)$.
We claim that $\{x\}\in h(C)$. Assume not; since $h(C)$ is finitary,
we can find a $z\in C$ such that $z\notin[x]_{h(C)}$, and then the transposition that swaps $x$ and $z$
would fix $C$ but move $h(C)$, which is a contradiction. Similarly, $\{y\}\in h(C)$.
Hence, the transposition that swaps $x$ and $y$ fixes $h(C)$,
but it moves $C$, contradicting that $h$ is injective.
\end{proof}

Now the next theorem immediately follows from Fact~\ref{sh12}, Lemma~\ref{sh25}, and the Jech--Sochor theorem.

\begin{theorem}\label{sh26}
It is consistent with $\mathsf{ZF}$ that there exists an infinite set $A$ such that
\begin{enumerate}[label=\upshape(\roman*)]
  \item $|\mathscr{B}(A)|$ and $|\mathrm{Part}_{\mathrm{fin}}(A)|$ are incomparable;
  \item $|\mathscr{B}(A)|$ and $|\mathscr{P}(A)|$ are incomparable;
  \item $|\mathscr{B}_{\mathrm{fin}}(A)|$ and $|\mathscr{P}(A)|$ are incomparable.
\end{enumerate}
\end{theorem}

As easily seen, for infinite well-orderable $A$, $|A|=|\mathrm{fin}(A)|=|\mathscr{B}_{\mathrm{fin}}(A)|$
and $|\mathscr{P}(A)|=|\mathscr{B}(A)|=|\mathrm{Part}_{\mathrm{fin}}(A)|=|\mathrm{Part}(A)|=|\mathscr{P}(A^2)|$.
Therefore, Theorems~\ref{sh18}, \ref{sh24}, and~\ref{sh26} show that Diagram~\ref{sh15} is optimal
in the sense that the $\leqslant$, $<$, or $\neq$ relations between the cardinalities
of these sets not indicted in the diagram cannot be proved in~$\mathsf{ZF}$.
However, we do not know whether $|\mathrm{Part}_{\mathrm{fin}}(A)|<|\mathscr{B}(A)|$
for some infinite set $A$ is consistent with~$\mathsf{ZF}$.

\section{Theorems in $\mathsf{ZF}$}\label{sh00}
In this section, we prove in $\mathsf{ZF}$ some results concerning $\mathscr{B}(A)$,
as well as the inequalities $|\mathscr{B}_{\mathrm{fin}}(A)|<|\mathrm{Part}_{\mathrm{fin}}(A)|$
and $|\mathscr{B}_{\mathrm{fin}}(A)|\neq|\mathscr{P}(A)|$ indicted in Diagram~\ref{sh15}.

Although Cantor's theorem may fail for $\mathscr{B}(A)$,
it does hold under the existence of auxiliary functions.

\begin{definition}
Let $f$ be a function on $A$. An \emph{auxiliary function} for $f$ is a function $g$ defined on $\ran(f)$
such that, for all $z\in\ran(f)$, $g(z)$ is a finitary partition of $f^{-1}[\{z\}]$
with at least one non-singleton block.
\end{definition}

\begin{lemma}\label{sh27}
From a function $f:A\to\mathscr{B}(A)$ and an auxiliary function for~$f$,
one can explicitly define a finitary partition of $A$ not in $\ran(f)$.
\end{lemma}
\begin{proof}
Use Cantor's diagonal construction.
Let $f$ be a function from $A$ to $\mathscr{B}(A)$ and let $g$ be an auxiliary function for $f$.
Let $h$ be the function on $\ran(f)$ defined by
\[
h(P)=
\begin{cases}
\{\{z\}\mid f(z)=P\} & \text{if $x\sim_Py$ for some distinct $x,y\in f^{-1}[\{P\}]$,}\\
g(P)                 & \text{otherwise.}
\end{cases}
\]
Then $\bigcup_{P\in\ran(f)}h(P)$ is a finitary partition of $A$ not in $\ran(f)$.
\end{proof}

\subsection{A general result}
We prove a general result which states that, if $\mathscr{B}(A)$ is Dedekind infinite,
then there are no Dedekind finite-to-one functions from $\mathscr{B}(A)$ to $\fin(A)$.

\begin{lemma}\label{sh28}
For any infinite ordinal $\alpha$, one can explicitly define an injection $f:\alpha\times\alpha\to\alpha$.
\end{lemma}
\begin{proof}
See~\cite[2.1]{Specker1954}.
\end{proof}

\begin{lemma}\label{sh29}
From an infinite ordinal $\alpha$, one can explicitly define an injection $f:\fin(\alpha)\to\alpha$.
\end{lemma}
\begin{proof}
See~\cite[Theorem~5.19]{Halbeisen2017}.
\end{proof}

\begin{lemma}\label{sh30}
From a finite-to-one function $f:\alpha\to A$, where $\alpha$ is an infinite ordinal,
one can explicitly define an injection $g:\alpha\to A$.
\end{lemma}
\begin{proof}
See~\cite[Lemma~3.3]{Shen2017}.
\end{proof}

The main idea of the following proof is originally from~\cite[Theorem~3]{HalbeisenShelah1994}.

\begin{lemma}\label{sh31}
From an injection $f:\alpha\to\fin(A)$, where $\alpha$ is an infinite ordinal,
one can explicitly define an injection $h:\alpha\to\fin(A)$ such that
the sets in $\ran(h)$ are pairwise disjoint and contain at least two elements.
\end{lemma}
\begin{proof}
Let $f$ be an injection from $\alpha$ into $\fin(A)$, where $\alpha$ is an infinite ordinal.
Let $\sim$ be the equivalence relation on $A$ defined by
\[
x\sim y\quad\text{if and only if}\quad\forall\beta<\alpha(x\in f(\beta)\leftrightarrow y\in f(\beta)).
\]
Clearly, for every $x\in\bigcup\ran(f)$, the equivalence class $[x]_\sim$ is finite.

We want to show that there are $\alpha$ many such equivalence classes.
In order to prove this, define a function $\Psi$ on $\bigcup\ran(f)$ by
\[
\Psi(x)=\bigl\{\gamma<\alpha\bigm|x\in f(\gamma)\text{ and }
\textstyle\bigcap\{f(\beta)\mid\beta<\gamma\text{ and }x\in f(\beta)\}\nsubseteq f(\gamma)\bigr\}.
\]

We claim that, for all $x,y\in\bigcup\ran(f)$,
\begin{equation}\label{sh32}
x\sim y\quad\text{if and only if}\quad\Psi(x)=\Psi(y).
\end{equation}
Clearly, if $x\sim y$ then $\Psi(x)=\Psi(y)$.
For the other direction, assume towards a contradiction that $\Psi(x)=\Psi(y)$ but not $x\sim y$.
Let $\delta<\alpha$ be the least ordinal such that $x\in f(\delta)$ is not equivalent to $y\in f(\delta)$.
Without loss of generality, assume that $x\in f(\delta)$ but $y\notin f(\delta)$.
Since $y\notin f(\delta)$, we have $\delta\notin\Psi(y)=\Psi(x)$, which implies that
$\bigcap\{f(\beta)\mid\beta<\delta\text{ and }x\in f(\beta)\}\subseteq f(\delta)$.
Since, for all $\beta<\delta$, $x\in f(\beta)$ if and only if $y\in f(\beta)$,
it follows that $y\in\bigcap\{f(\beta)\mid\beta<\delta\text{ and }x\in f(\beta)\}\subseteq f(\delta)$,
which is a contradiction.

We also claim that, for all $x\in\bigcup\ran(f)$,
\begin{equation}\label{sh33}
\Psi(x)\in\fin(\alpha).
\end{equation}
Let $\xi$ be the least ordinal such that $x\in f(\xi)$.
Then $\xi$ is the first element of $\Psi(x)$.
For all $\gamma,\delta\in\Psi(x)$ with $\gamma<\delta$,
if $f(\xi)\cap f(\gamma)=f(\xi)\cap f(\delta)$, then
$\bigcap\{f(\beta)\mid\beta<\delta\text{ and }x\in f(\beta)\}\subseteq f(\xi)\cap f(\gamma)\subseteq f(\delta)$,
contradicting that $\delta\in\Psi(x)$.
Hence, the function that maps each $\gamma\in\Psi(x)$ to $f(\xi)\cap f(\gamma)$
is an injection from $\Psi(x)$ into $\mathscr{P}(f(\xi))$.
Since $\mathscr{P}(f(\xi))$ is finite, it follows that $\Psi(x)$ is finite.

Now, by Lemma~\ref{sh29}, we can explicitly define an injection $p:\fin(\alpha)\to\alpha$.
By \eqref{sh33}, $\ran(\Psi)\subseteq\fin(\alpha)$.
Let $R$ be the well-ordering of $\ran(\Psi)$ induced by~$p$;
that is, $R=\{(a,b)\mid a,b\in\ran(\Psi)\text{ and }p(a)<p(b)\}$.
Let $\theta$ be the order type of $\langle\ran(\Psi),R\rangle$,
and let $\Theta$ be the unique isomorphism of $\langle\ran(\Psi),R\rangle$ onto $\langle\theta,\in\rangle$.
It is easy to see that $\theta$ is an infinite ordinal.

Again, by Lemma~\ref{sh29}, we can explicitly define an injection $q:\fin(\theta)\to\theta$.
By \eqref{sh32}, the function that maps each $\beta<\alpha$ to $\Psi[f(\beta)]$
is an injection from $\alpha$ into $\fin(\ran(\Psi))$.
Let $g$ be the function on $\alpha$ defined by
\[
g(\beta)=\Theta^{-1}(q(\Theta[\Psi[f(\beta)]])).
\]
$g$ is visualized by the following diagram:
\[
\begin{matrix}
g: & \alpha & \to     & \fin(\ran(\Psi)) & \to     & \fin(\theta)           & \to     & \theta                    & \to     & \ran(\Psi)\\
   & \beta  & \mapsto & \Psi[f(\beta)]   & \mapsto & \Theta[\Psi[f(\beta)]] & \mapsto & q(\Theta[\Psi[f(\beta)]]) & \mapsto & g(\beta).
\end{matrix}
\]
Hence, $g$ is an injection from $\alpha$ into $\ran(\Psi)$.

By Lemma~\ref{sh28}, we can explicitly define an injection $s:\alpha\times\alpha\to\alpha$.
Then the function $h$ on $\alpha$ defined by
\[
h(\beta)=\Psi^{-1}[\{g(s(\beta,0)),g(s(\beta,1))\}]
\]
is the required function.
\end{proof}

\begin{corollary}\label{sh34}
From an injection $f:\alpha\to\fin(A)$, where $\alpha$ is an infinite ordinal,
one can explicitly define a surjection $g:A\twoheadrightarrow\alpha$ and an auxiliary function for $g$.
\end{corollary}
\begin{proof}
Let $f$ be an injection from $\alpha$ into $\fin(A)$, where $\alpha$ is an infinite ordinal.
By Lemma~\ref{sh31}, we can explicitly define an injection $h:\alpha\to\fin(A)$ such that
the sets in $\ran(h)$ are pairwise disjoint and contain at least two elements.
Now, the function $g$ on $A$ defined by
\[
g(x)=
\begin{cases}
\text{the unique }\beta<\alpha\text{ for which }x\in h(\beta) & \text{if $x\in\bigcup\ran(h)$,}\\
0                                                             & \text{otherwise,}
\end{cases}
\]
is a surjection from $A$ onto $\alpha$, and the function $t$ on $\alpha$ defined by
\[
t(\beta)=
\begin{cases}
\{h(0)\}\cup\{\{x\}\mid x\in A\setminus\bigcup\ran(h)\} & \text{if $\beta=0$,}\\
\{h(\beta)\}                                            & \text{otherwise,}
\end{cases}
\]
is an auxiliary function for $g$.
\end{proof}

Now we are ready to prove our main theorem.

\begin{theorem}\label{sh35}
If $\mathscr{B}(A)$ is Dedekind infinite, then there are no Dedekind finite-to-one functions from $\mathscr{B}(A)$ to $\fin(A)$.
\end{theorem}
\begin{proof}
Assume towards a contradiction that $\mathscr{B}(A)$ is Dedekind infinite
and there is a Dedekind finite-to-one function $\Phi:\mathscr{B}(A)\to\fin(A)$.
Let $h$ be an injection from $\omega$ into $\mathscr{B}(A)$.
In what follows, we get a contradiction by constructing by recursion
an injection $H$ from the proper class of ordinals into the set $\mathscr{B}(A)$.

For $n\in\omega$, take $H(n)=h(n)$.
Now, we assume that $\alpha$ is an infinite ordinal and $H\rstr\alpha$ is an injection from $\alpha$ into $\mathscr{B}(A)$.
Then $\Phi\circ(H\rstr\alpha)$ is a Dedekind finite-to-one function from $\alpha$ to $\fin(A)$.
Since all Dedekind finite subsets of $\alpha$ are finite, $\Phi\circ(H\rstr\alpha)$ is finite-to-one.
By Lemma~\ref{sh30}, $\Phi\circ(H\rstr\alpha)$ explicitly provides an injection $f:\alpha\to\fin(A)$.
Therefore, by Corollary~\ref{sh34}, from $f$, we can explicitly define a surjection $g:A\twoheadrightarrow\alpha$
and an auxiliary function $t$ for $g$. Then $(H\rstr\alpha)\circ g$ is a surjection from $A$ onto $H[\alpha]$
and $t\circ(H\rstr\alpha)^{-1}$ is an auxiliary function for $(H\rstr\alpha)\circ g$.
Hence, it follows from Lemma~\ref{sh27} that we can explicitly define an
$H(\alpha)\in\mathscr{B}(A)\setminus H[\alpha]$ from $H\rstr\alpha$ (and $\Phi$).
\end{proof}

We draw some corollaries from the above theorem.
The next two corollaries immediately follows from Theorem~\ref{sh35} and Facts~\ref{sh04}, \ref{sh05}, and~\ref{sh11}.

\begin{corollary}\label{sh36}
If $\mathscr{B}(A)$ is Dedekind infinite, then there are no Dedekind finite-to-one functions from $\mathscr{B}(A)$ to $\seqi(A)$.
\end{corollary}

\begin{corollary}\label{sh37}
If $\mathscr{B}(A)$ is Dedekind infinite, then there are no Dedekind finite-to-one functions from $\mathscr{B}(A)$ to $\scrBf(A)$,
and thus $|\mathscr{B}_{\mathrm{fin}}(A)|<|\mathscr{B}(A)|$.
\end{corollary}

\begin{corollary}\label{sh38}
If $\mathscr{B}(A)\neq\scrBf(A)$, then $|\mathscr{B}_{\mathrm{fin}}(A)|<|\mathscr{B}(A)|$.
\end{corollary}
\begin{proof}
Suppose $\mathscr{B}(A)\neq\scrBf(A)$. Clearly, $|\mathscr{B}_{\mathrm{fin}}(A)|\leqslant|\mathscr{B}(A)|$.
If $|\mathscr{B}(A)|=|\mathscr{B}_{\mathrm{fin}}(A)|$, since $\scrBf(A)\subset\mathscr{B}(A)$,
it follows that $\mathscr{B}(A)$ is Dedekind infinite, contradicting Corollary~\ref{sh37}.
Hence, $|\mathscr{B}_{\mathrm{fin}}(A)|<|\mathscr{B}(A)|$.
\end{proof}

The following corollary is also proved in~\cite[Theorem~3.7]{PhansamdaengVejjajiva2022}.

\begin{corollary}\label{sh39}
For all infinite sets $A$, $|\mathrm{fin}(A)|<|\mathscr{B}(A)|$.
\end{corollary}
\begin{proof}
By Fact~\ref{sh12}, $|\mathrm{fin}(A)|\leqslant|\mathscr{B}_{\mathrm{fin}}(A)|\leqslant|\mathscr{B}(A)|$.
If $|\mathscr{B}(A)|=|\mathrm{fin}(A)|$, since the injection constructed in the proof of Fact~\ref{sh12} is not surjective,
it follows that $\mathscr{B}(A)$ is Dedekind infinite, contradicting Theorem~\ref{sh35}. Thus, $|\mathrm{fin}(A)|<|\mathscr{B}(A)|$.
\end{proof}

\subsection{$|\mathscr{B}(A)|\neq|\mathrm{seq}(A)|$}
We need the following result, which is a Kuratowski-like theorem for $\mathscr{B}(A)$.

\begin{theorem}\label{sh40}
For all sets $A$, the following are equivalent:
\begin{enumerate}[label=\upshape(\roman*)]
  \item $\mathscr{B}(A)$ is Dedekind infinite;
  \item $\ns(P)$ is power Dedekind infinite for some $P\in\mathscr{B}(A)$;
  \item $\mathscr{P}(\omega)\preccurlyeq\mathscr{B}(A)$.
\end{enumerate}
\end{theorem}
\begin{proof}
(i)${}\Rightarrow{}$(ii).
Suppose that $\mathscr{B}(A)$ is Dedekind infinite.
Assume towards a contradiction that $\ns(P)$ is power Dedekind finite for all $P\in\mathscr{B}(A)$.
Let $\langle P_k\mid k\in\omega\rangle$ be a denumerable family of finitary partitions of $A$.
We define by recursion a sequence $\langle Q_n\mid n\in\omega\rangle$ of finitary partitions of $A$ such that
$\ns(Q_j)\neq\varnothing$ and $\bigcup\ns(Q_i)\cap\bigcup\ns(Q_j)=\varnothing$ whenever $i\neq j$ as follows.

Let $n\in\omega$ and assume that $Q_i$ has already been defined for all $i<n$.
Let $B=\bigcup_{i<n}\bigcup\ns(Q_i)$. By assumption, $\ns(Q_i)$ is power Dedekind finite,
so is $\bigcup\ns(Q_i)$ by Fact~\ref{sh01}. Hence, $B$ is power Dedekind finite. Consider the following two cases:

\textsc{Case}~1.
For a least $k\in\omega$, $x\sim_{P_k}y$ for some distinct $x,y\in A\setminus B$. Then we define
\[
Q_n=\{[z]_{P_k}\setminus B\mid z\in A\setminus B\}\cup\{\{z\}\mid z\in B\}.
\]
It is easy to see that $Q_n$ is a finitary partition of $A$ such that $\ns(Q_n)\neq\varnothing$
and $\bigcup\ns(Q_i)\cap\bigcup\ns(Q_n)=\varnothing$ for all $i<n$.

\textsc{Case}~2.
Otherwise. Then, for all $k\in\omega$ and all $x\in\bigcup\ns(P_k)\setminus B$,
\begin{equation}\label{sh41}
[x]_{P_k}\cap B\neq\varnothing\text{ and }[x]_{P_k}\setminus B=\{x\}.
\end{equation}
Define by recursion a sequence $\langle k_m\mid m\in\omega\rangle$ of natural numbers as follows.

Let $m\in\omega$ and assume that $k_j$ has been defined for $j<m$.
By assumption, $\ns(P_{k_j})$ is power Dedekind finite, so is $\bigcup\ns(P_{k_j})$ by Fact~\ref{sh01}.
Therefore, $B\cup\bigcup_{j<m}\bigcup\ns(P_{k_j})$ is power Dedekind finite.
Hence, there is a $k\in\omega$ such that $\ns(P_k)\nsubseteq\mathscr{P}(B\cup\bigcup_{j<m}\bigcup\ns(P_{k_j}))$.
Define $k_m$ to be the least such $k$. Then
\begin{equation}\label{sh42}
{\textstyle\bigcup}\ns(P_{k_m})\nsubseteq B\cup\bigcup_{j<m}{\textstyle\bigcup}\ns(P_{k_j}).
\end{equation}

Let $f$ and $g$ be the functions on $\omega$ defined by
\begin{align*}
f(m) & ={\textstyle\bigcup}\ns(P_{k_m})\setminus\bigl(B\cup\bigcup_{j<m}{\textstyle\bigcup}\ns(P_{k_j})\bigr)\\
g(m) & =\{[x]_{P_{k_m}}\cap B\mid x\in f(m)\}
\end{align*}
By~\eqref{sh41}, for every $m\in\omega$, $g(m)$ is a finitary partition of a subset of $B$,
and hence $\sim_{g(m)}$ is a subset of $\mathscr{P}(B^2)$.
Since $\mathscr{P}(B^2)$ is Dedekind finite by~Fact~\ref{sh02},
there are least $l_0,l_1\in\omega$ with $l_0<l_1$ such that ${\sim_{g(l_0)}}={\sim_{g(l_1)}}$,
and thus $g(l_0)=g(l_1)$. Let
\[
D=\bigl\{\{x,y\}\bigm|x\in f(l_0),y\in f(l_1)\text{ and }[x]_{P_{k_{l_0}}}\cap B=[y]_{P_{k_{l_1}}}\cap B\bigr\}.
\]
Since $l_0<l_1$, $f(l_0)\cap f(l_1)=\varnothing$, and thus, by~\eqref{sh41},
the sets in $D$ are pairwise disjoint. By~\eqref{sh42}, $f(m)\neq\varnothing$ for all $m\in\omega$,
and since $g(l_0)=g(l_1)$, it follows that $D\neq\varnothing$. Now, we define
\[
Q_n=D\cup\{\{z\}\mid z\in A\setminus\textstyle\bigcup D\}.
\]
Then $Q_n$ is a finitary partition of $A$ such that $\ns(Q_n)=D\neq\varnothing$;
since $B\cap\bigcup D=\varnothing$, it follows that $\bigcup\ns(Q_i)\cap\bigcup\ns(Q_n)=\varnothing$ for all $i<n$.

Finally,
\[
Q=\bigcup_{n\in\omega}\ns(Q_n)\cup\bigl\{\{z\}\bigm|z\in A\setminus\bigcup_{n\in\omega}{\textstyle\bigcup}\ns(Q_n)\bigr\}
\]
is a finitary partition of $A$ such that $\ns(Q)=\bigcup_{n\in\omega}\ns(Q_n)$ is power Dedekind infinite,
which is a contradiction.

(ii)${}\Rightarrow{}$(iii).
Suppose that $P$ is a finitary partition of $A$ such that $\ns(P)$ is power Dedekind infinite.
Let $p$ be a surjection from $\ns(P)$ onto $\omega$. Then the function $h$ on $\mathscr{P}(\omega)$ defined by
\[
h(u)=p^{-1}[u]\cup\{\{z\}\mid z\in A\setminus\textstyle\bigcup p^{-1}[u]\}
\]
is an injection from $\mathscr{P}(\omega)$ into $\mathscr{B}(A)$.

(iii)${}\Rightarrow{}$(i). Obviously.
\end{proof}

\begin{corollary}\label{sh43}
If $\mathscr{B}(A)$ is Dedekind infinite, then there are no Dedekind finite-to-one functions from $\mathscr{B}(A)$ to $\seq(A)$.
\end{corollary}
\begin{proof}
Assume towards a contradiction that $\mathscr{B}(A)$ is Dedekind infinite
and there exists a Dedekind finite-to-one function from $\mathscr{B}(A)$ to $\seq(A)$.
If $A$ is Dedekind infinite, then $\seq(A)\approx\seqi(A)$ by Fact~\ref{sh08},
contradicting Corollary~\ref{sh36}. Otherwise, by Fact~\ref{sh07},
there is a Dedekind finite-to-one function from $\seq(A)$ to $\omega$.
By Theorem~\ref{sh40}, $\mathscr{B}(\omega)\approx\mathscr{P}(\omega)\preccurlyeq\mathscr{B}(A)$,
and therefore there is a Dedekind finite-to-one function from $\mathscr{B}(\omega)$ to $\omega$,
contradicting again Corollary~\ref{sh36}.
\end{proof}

\begin{corollary}\label{sh44}
For all non-empty sets $A$, $|\mathscr{B}(A)|\neq|\mathrm{seq}(A)|$.
\end{corollary}
\begin{proof}
For all non-empty sets $A$, if $|\mathscr{B}(A)|=|\mathrm{seq}(A)|$,
then it follows from Fact~\ref{sh06} that $\mathscr{B}(A)$ is Dedekind infinite,
contradicting Corollary~\ref{sh43}.
\end{proof}

We do not know whether $|\mathscr{B}(A)|\neq|\mathrm{seq}^{\text{1-1}}(A)|$
for all infinite sets $A$ is provable in~$\mathsf{ZF}$.

\subsection{$|A^n|<|\mathscr{B}(A)|$}
The main idea of the following proof is originally from~\cite[Lemma(i)]{Truss1973} (cf.~also~\cite{Shen2023b}).

\begin{lemma}\label{sh45}
For all $n\in\omega$, if $|A|\geqslant 2n(2^{n+1}-1)$, then $|A^n|\leqslant|\mathscr{B}_{\mathrm{fin}}(A)|$.
\end{lemma}
\begin{proof}
Let $n\in\omega$ and let $A$ be a set with at least $2n(2^{n+1}-1)$ elements.
Since $2^{i}=1+\sum_{k<i}2^k$, we can choose $2n(2^{n+1}-1)$ distinct elements of $A$ so that they are divided into
$n+1$ sets $H_i$ ($i\leqslant n$) with
\[
H_i=\{a_{i,j}\mid j<2n\}\cup\{b_{i,x}\mid x\in H_k\text{ for some }k<i\}.
\]
We construct an injection $f$ from $A^n$ into $\scrBf(A)$ as follows.
Without loss of generality, assume that $A\cap\omega=\varnothing$.

Let $s\in A^n$. Let $i_s$ be the least $i\leqslant n$ for which $\mathrm{ran}(s)\cap H_i=\varnothing$.
There is such an $i$ because $|\mathrm{ran}(s)|\leqslant n$. Let $t_s$ be the function on $n$ defined by
\[
t_s(j)=
\begin{cases}
s(j)                      & \text{if $s(j)\neq s(k)$ for all $k<j$,}\\
\max\{k<j\mid s(j)=s(k)\} & \text{otherwise.}
\end{cases}
\]
Clearly, $t_s\in\seqi(A\cup n)$. Let $u_s$ be the function on $n$ defined by
\[
u_s(j)=
\begin{cases}
a_{i_s,n+t_s(j)} & \text{if $t_s(j)\in n$,}\\
b_{i_s,t_s(j)}   & \text{if $t_s(j)\in H_k$ for some $k<i_s$,}\\
t_s(j)           & \text{otherwise.}
\end{cases}
\]
Then it is easy to see that $u_s\in\seqi(A)$. Now, we define
\[
f(s)=\bigl\{\{a_{i_s,j},u_s(j)\}\bigm|j<n\bigr\}\cup\bigl\{\{z\}\bigm|z\in A\setminus(\{a_{i_s,j}\mid j<n\}\cup\ran(u_s))\bigr\}.
\]
Clearly, $f(s)\in\scrBf(A)$. We prove that $f$ is injective by showing that $s$ is uniquely determined by $f(s)$ in the following way.

First, $i_s$ is the least $i\leqslant n$ such that $H_i\cap\bigcup\ns(f(s))\neq\varnothing$.
Second, $u_s$ is the function on $n$ such that $\{a_{i_s,j},u_s(j)\}\in f(s)$ for all $j<n$.
Then, $t_s$ is the function on $n$ such that, for every $j<n$, either $t_s(j)$ is the unique
element of $n$ for which $u_s(j)=a_{i_s,n+t_s(j)}$, or $t_s(j)$ is the unique element of
$\bigcup_{k<i_s}H_k$ for which $u_s(j)=b_{i_s,t_s(j)}$, or $t_s(j)=u_s(j)\notin H_{i_s}$.
Finally, $s$ is the function on $n$ recursively determined by
\[
s(j)=
\begin{cases}
t_s(j)    & \text{if $t_s(j)\in A$,}\\
s(t_s(j)) & \text{otherwise.}
\end{cases}
\]

Hence, $f$ is an injection from $A^n$ into $\scrBf(A)$.
\end{proof}

\begin{corollary}\label{sh46}
For all $n\in\omega$ and all infinite sets $A$, $|A^n|<|\mathscr{B}(A)|$.
\end{corollary}
\begin{proof}
By Lemma~\ref{sh45}, $|A^n|\leqslant|\mathscr{B}_{\mathrm{fin}}(A)|\leqslant|\mathscr{B}(A)|$.
If $|\mathscr{B}(A)|=|A^n|$, since the injection constructed in the proof of Lemma~\ref{sh45} is not surjective,
$\mathscr{B}(A)$ is Dedekind infinite, contradicting Corollary~\ref{sh43}. Thus, $|A^n|<|\mathscr{B}(A)|$.
\end{proof}

\begin{lemma}\label{sh47}
If $A$ is Dedekind infinite, then $|\mathrm{seq}(A)|\leqslant|\mathscr{B}_{\mathrm{fin}}(A)|$.
\end{lemma}
\begin{proof}
Let $h$ be an injection from $\omega$ into $A$.
By the proof of Lemma~\ref{sh45}, from $h\rstr(2n(2^{n+1}-1))$,
we can explicitly define an injection
\[
f_n:A^n\to\{P\in\mathscr{B}(A)\mid|\mathrm{ns}(P)|=n\}.
\]
Then $\bigcup_{n\in\omega}f_n$ is an injection from $\seq(A)$ into $\scrBf(A)$.
\end{proof}

The following corollary immediately follows from Lemma~\ref{sh47} and Corollary~\ref{sh44}.

\begin{corollary}\label{sh48}
For all Dedekind infinite sets $A$, $|\mathrm{seq}(A)|<|\mathscr{B}(A)|$.
\end{corollary}

\subsection{A Cantor-like theorem for $\mathscr{B}(A)$}
Under the assumption that there is a finitary partition of $A$ without singleton blocks,
we show that Cantor's theorem holds for $\mathscr{B}(A)$.
The key step of our proof is the following lemma.

\begin{lemma}\label{sh49}
From a finitary partition $P$ of $A$ without singleton blocks and a surjection $f:A\twoheadrightarrow\alpha$,
where $\alpha$ is an infinite ordinal, one can explicitly define a surjection $g:A\twoheadrightarrow\alpha$
and an auxiliary function for $g$.
\end{lemma}
\begin{proof}
Let $P$ be a finitary partition of $A$ without singleton blocks,
and let $f$ be a surjection from $A$ onto $\alpha$,
where $\alpha$ is an infinite ordinal. Let
\[
Q=\{f[E]\mid E\in P\}.
\]
Clearly, $Q\subseteq\fin(\alpha)$. By Lemma~\ref{sh29},
we can explicitly define an injection $p:\mathrm{fin}(\alpha)\to\alpha$.
Let $R$ be the well-ordering of $Q$ induced by~$p$;
that is, $R=\{(a,b)\mid a,b\in Q\text{ and }p(a)<p(b)\}$.
Since $P$ is a partition of $A$, $Q$ is a cover of $\alpha$.
Define a function $h$ from $\alpha$ to $Q$ by setting, for $\beta\in\alpha$,
\[
h(\beta)=\text{the $R$-least }c\in Q\text{ such that }\beta\in c.
\]
Since $\beta\in h(\beta)$ and $h(\beta)$ is finite for all $\beta\in\alpha$,
$h$ is a finite-to-one function from $\alpha$ to $Q$.
Hence, by Lemma~\ref{sh30}, $h$ explicitly provides an injection from $\alpha$ into $Q$.
Since $p{\upharpoonright}Q$ is an injection from $Q$ into $\alpha$, it follows from
Theorem~\ref{cbt} that we can explicitly define a bijection $q$ between $Q$ and~$\alpha$.

Now, the function $g$ on $A$ defined by
\[
g(x)=q(f[[x]_P])
\]
is a surjection from $A$ onto $\alpha$, and the function $t$ on $\alpha$ defined by
\[
t(\beta)=\{E\in P\mid q(f[E])=\beta\}
\]
is an auxiliary function for $g$.
\end{proof}

\begin{theorem}\label{sh50}
For all infinite sets $A$, if there is a finitary partition of $A$ without singleton blocks,
then there are no surjections from $A$ onto $\mathscr{B}(A)$.
\end{theorem}
\begin{proof}
Let $A$ be an infinite set and let $P$ be a finitary partition of $A$ without singleton blocks.
Assume towards a contradiction that there is a surjection $\Phi:A\twoheadrightarrow\mathscr{B}(A)$.
By Corollary~\ref{sh13}, $\mathscr{B}(A)$ is power Dedekind infinite, so is~$A$.
Since $\bigcup\ns(P)=A$ is power Dedekind infinite, so is $\ns(P)$ by Fact~\ref{sh01}.
Hence, by Theorem~\ref{sh40}, $\mathscr{B}(A)$ is Dedekind infinite.
Let $h$ be an injection from $\omega$ into $\mathscr{B}(A)$.
In what follows, we get a contradiction by constructing by recursion
an injection $H$ from the proper class of ordinals into $\mathscr{B}(A)$.

For $n\in\omega$, take $H(n)=h(n)$.
Now, we assume that $\alpha$ is an infinite ordinal and $H\rstr\alpha$ is an injection from $\alpha$ into $\mathscr{B}(A)$.
Then $(H\rstr\alpha)^{-1}\circ\Phi$ is a surjection from a subset of $A$ onto $\alpha$
and thus can be extended by zero to a surjection $f:A\twoheadrightarrow\alpha$.
By Lemma~\ref{sh49}, from $P$ and $f$, we can explicitly define a surjection $g:A\twoheadrightarrow\alpha$
and an auxiliary function $t$ for $g$. Then $(H\rstr\alpha)\circ g$ is a surjection from $A$ onto $H[\alpha]$
and $t\circ(H\rstr\alpha)^{-1}$ is an auxiliary function for $(H\rstr\alpha)\circ g$.
Hence, it follows from Lemma~\ref{sh27} that we can explicitly define an
$H(\alpha)\in\mathscr{B}(A)\setminus H[\alpha]$ from $H\rstr\alpha$ (and $P,\Phi$).
\end{proof}

Under the same assumption, we also show that there are no finite-to-one functions from $\mathscr{B}(A)$ to $A^n$.

\begin{theorem}\label{sh51}
For all infinite sets $A$, if there is a finitary partition of~$A$ without singleton blocks,
then there are no finite-to-one functions from $\mathscr{B}(A)$ to $A^n$ for every $n\in\omega$.
\end{theorem}
\begin{proof}
Let $A$ be an infinite set and let $P$ be a finitary partition of $A$ without singleton blocks.
Assume towards a contradiction that there is a finite-to-one function from $\mathscr{B}(A)$ to $A^n$ for some $n\in\omega$.
By Corollary~\ref{sh13}, $\mathscr{B}(A)$ is power Dedekind infinite, so is $A$ by Facts~\ref{sh01} and~\ref{sh02}.
Since $\bigcup\ns(P)=A$ is power Dedekind infinite, so is $\ns(P)$ by Fact~\ref{sh01}.
Hence, by Theorem~\ref{sh40}, $\mathscr{B}(A)$ is Dedekind infinite, contradicting Corollary~\ref{sh43}.
\end{proof}

\subsection{The inequalities $|\mathscr{B}_{\mathrm{fin}}(A)|<|\mathrm{Part}_{\mathrm{fin}}(A)|$
and $|\mathscr{B}_{\mathrm{fin}}(A)|\neq|\mathscr{P}(A)|$}

\begin{lemma}\label{sh54}
If $A$ is power Dedekind infinite, then there are no finite-to-one functions from $\mathscr{P}(A)$ to $\fin(A)$.
\end{lemma}
\begin{proof}
See~\cite[Corollary~3.7]{Shen2017}.
\end{proof}

The next corollary immediately follows from Lemma~\ref{sh54} and Facts~\ref{sh05} and~\ref{sh11}.

\begin{corollary}\label{sh55}
If $A$ is power Dedekind infinite, then there are no finite-to-one functions from $\mathscr{P}(A)$ to $\scrBf(A)$.
\end{corollary}

\begin{theorem}\label{sh56}
For all infinite sets $A$, $|\mathscr{B}_{\mathrm{fin}}(A)|<|\mathrm{Part}_{\mathrm{fin}}(A)|$.
\end{theorem}
\begin{proof}
Let $A$ be an infinite set. By Fact~\ref{sh14},
$|\mathscr{B}_{\mathrm{fin}}(A)|\leqslant|\mathrm{Part}_{\mathrm{fin}}(A)|$.
Assume $|\mathrm{Part}_{\mathrm{fin}}(A)|=|\mathscr{B}_{\mathrm{fin}}(A)|$.
Since the injection constructed in the proof of Fact~\ref{sh14} is not surjective,
$\Partf(A)$ is Dedekind infinite, and thus $A$ is power Dedekind infinite by Corollary~\ref{sh10}.
Now, by Fact~\ref{sh12}, $|\mathscr{P}(A)|\leqslant|\mathrm{Part}_{\mathrm{fin}}(A)|=|\mathscr{B}_{\mathrm{fin}}(A)|$,
contradicting Corollary~\ref{sh55}. Hence, $|\mathscr{B}_{\mathrm{fin}}(A)|<|\mathrm{Part}_{\mathrm{fin}}(A)|$.
\end{proof}

Finally, we prove $|\mathscr{B}_{\mathrm{fin}}(A)|\neq|\mathscr{P}(A)|$.
For this, we need the following number-theoretic lemma.

\begin{lemma}\label{sh52}
For each $n\in\omega$, let $B_n$ be the $n$-th Bell number; that is, $B_n=|\mathscr{B}(n)|$.
Then $B_m$ is not a power of $2$ for all $m\geqslant3$.
\end{lemma}
\begin{proof}
Let $m\geqslant3$. Then $B_m>4$.
It suffices to prove that $B_m$ is not divisible by $8$.
By~\cite[Theorem~6.4]{Lunnon1979}, $B_{n+24}\equiv B_n\pmod{8}$ for all $n\in\omega$.
But $B_n$ modulo $8$ for $n$ from $0$ to $23$ are
\[
1,1,2,5,7,4,3,5,4,3,7,2,5,5,2,1,3,4,7,1,4,7,3,2.
\]
Hence, $B_m$ is not divisible by $8$.
\end{proof}

\begin{lemma}\label{sh53}
If $\fin(\mathscr{P}(A))$ is Dedekind infinite, then $A$ is power Dedekind infinite.
\end{lemma}
\begin{proof}
See~\cite[Theorem~3.2]{Shen2023a}.
\end{proof}

\begin{theorem}\label{sh57}
For all non-empty sets $A$, $|\mathscr{P}(A)|\neq|\mathscr{B}_{\mathrm{fin}}(A)|$.
\end{theorem}
\begin{proof}
If $A$ is a singleton, $|\mathscr{P}(A)|=2\neq1=|\mathscr{B}_{\mathrm{fin}}(A)|$.
Suppose $|A|\geqslant2$, and fix two distinct elements $a,b$ of $A$.
Assume toward a contradiction that there is a bijection $\Phi$ between $\mathscr{P}(A)$ and $\scrBf(A)$.
We define by recursion an injection $f$ from $\omega$ into $\fin(\mathscr{P}(A))$ as follows.

Take $f(0)=\{\{a\},\{b\}\}$. Let $n\in\omega$, and assume that $f(0),\dots,f(n)$ have been defined and
are pairwise distinct elements of $\fin(\mathscr{P}(A))$. Let $\sim$ be the equivalence relation on $A$ defined by
\[
x\sim y\quad\text{if and only if}\quad\forall C\in f(0)\cup\dots\cup f(n)(x\in C\leftrightarrow y\in C).
\]
Since $f(0),\dots,f(n)$ are finite, the quotient set $A/{\sim}$ is a finite partition of~$A$.
Let $k=|A/{\sim}|$ and let $U=\{\bigcup W\mid W\subseteq A/{\sim}\}$. Then $|U|=2^k$ and
\begin{equation}\label{sh58}
f(0)\cup\dots\cup f(n)\subseteq U.
\end{equation}
Since $\{a\},\{b\}\in A/{\sim}$, we have $k\geqslant2$. Let $D=\bigcup\{\bigcup\ns(P)\mid P\in\Phi[U]\}$.
Since $U$ is finite and $\Phi[U]\subseteq\scrBf(A)$, it follows that $D$ is finite.
Let $m=|D|$ and let $E=\{P\in\scrBf(A)\mid\bigcup\ns(P)\subseteq D\}$. Then $|E|=B_m$ and $\Phi[U]\subseteq E$.
Hence, $2^k=|U|=|\Phi[U]|\leqslant|E|=B_m$. Since $k\geqslant2$, we have $m\geqslant3$,
which implies that $B_m\neq2^k$ by Lemma~\ref{sh52}, and hence $\Phi[U]\subset E$.
Now, we define $f(n+1)=\Phi^{-1}[E\setminus\Phi[U]]$. Then $f(n+1)$ is a non-void finite subset of $\mathscr{P}(A)$.
By~\eqref{sh58}, it follows that $f(n+1)$ is disjoint from each of $f(0),\dots,f(n)$, and thus is distinct from each of them.

The existence of the above injection $f$ shows that $\fin(\mathscr{P}(A))$ is Dedekind infinite,
which implies that, by Lemma~\ref{sh53}, $A$ is power Dedekind infinite, contradicting Corollary~\ref{sh55}.
\end{proof}

\section{Open questions}
We conclude the paper with the following five open questions.

\begin{question}
Are the following statements consistent with~$\mathsf{ZF}$?
\begin{enumerate}
\item There exist an infinite set $A$ and a finite-to-one function from $\mathscr{B}(A)$ to $A$.
\item There exists an infinite set $A$ for which $|\mathscr{B}(A)|<|\mathcal{S}_3(A)|$,
      where $\mathcal{S}_3(A)$ is the set of permutations of $A$ with exactly $3$ non-fixed points.
\item There exists an infinite set $A$ for which $|\mathscr{B}(A)|=|\mathrm{seq}^{\text{1-1}}(A)|$.
\item There exists an infinite set $A$ for which $|\mathrm{Part}_{\mathrm{fin}}(A)|<|\mathscr{B}(A)|$.
\item There exist an infinite set $A$ and a surjection from $A^2$ onto $\Part(A)$.
\end{enumerate}
\end{question}

\subsection*{Acknowledgements}
I should like to give thanks to Professor Ira Gessel for providing a proof of Lemma~\ref{sh52}.
The author was partially supported by National Natural Science Foundation of China grant number 12101466.


\end{document}